\documentclass[12pt]{article}
\newcommand{\lcm}{\mathop{\mathrm{lcm}}}
\newtheorem{theorem}{Theorem}
\newtheorem{prop}[theorem]{Proposition}
\newtheorem{unnumber}{}

\newtheorem{prepf}[unnumber]{Proof}
\newenvironment{proof}{\prepf\rm}{\endprepf}

\begin{document}

\title{Enhanced power graphs of groups are weakly perfect}
\author{Peter J. Cameron\footnote{School of Mathematics and Statistics,
University of St Andrews, UK; email \texttt{pjc20@st-andrews.ac.uk}}\ \ and
Veronica Phan\footnote{37 Street 2, Ward 6, District 8, Ho Chi Minh City,
Vietnam; email \texttt{kyubivulpes@gmail.com}}}
\date{}
\maketitle

\begin{abstract}
A graph is weakly perfect if its clique number and chromatic number are equal.
We show that the enhanced power graph of a finite group $G$ is weakly perfect:
its clique number and chromatic number are equal to the maximum order of an
element of $G$. The proof requires a combinatorial lemma. We give some remarks
about related graphs.
\end{abstract}

\section{Introduction}

The \emph{directed power graph} of a finite group $G$, defined in~\cite{kq},
has the elements of $G$ as vertices, with an arc from $x$ to $y$ if $y=x^n$ for
some integer $n$. This relation is reflexive and transitive, hence is a
\emph{partial preorder}. The (undirected) \emph{power graph}, defined in
\cite{cgs}, is obtained by ignoring directions: that is, $x$ and $y$ are
joined if one is a power of the other. This graph is thus the comparability
graph of a partial preorder; a small extension of Dilworth's theorem shows
that it is \emph{perfect}, that is, every induced subgraph has clique number
equal to chromatic number.

Both these graphs were first defined for semigroups, but most work on them
has concerned groups.

According to the \emph{strong perfect graph theorem}~\cite{pgt}, a graph is
perfect if and only if it has no induced subgraph which is a cycle of odd
length greater than~$3$ or the complement of one.

The \emph{enhanced power graph} of $G$, defined in \cite{aetal}, again has
vertex set $G$, with $x$ and $y$ joined if there is an element $z$ such
that both $x$ and $y$ are powers of $z$. (Equivalently, $x$ and $y$ are joined
if and only if the group they generate is cyclic.) It is shown in
\cite{c:survey} that enhanced power graphs of finite groups are
\emph{universal}, that is, every finite graph occurs as an induced subgraph
of such a graph. Thus, these graphs are not in general perfect.

Our purpose here is to show that enhanced power graphs are \emph{weakly
perfect}, that is, they have chromatic number equal to clique number. Indeed
our result is not restricted to finite groups, but applies to groups in
which all elements have finite and bounded order.

\begin{theorem}
Let $G$ be a finite group, or a torsion group of bounded exponent. Then the
clique number and the chromatic number of $G$ are both equal to the maximal
order of an element of $G$.
\label{t:main}
\end{theorem}

The result for clique number is known, and the proof is straightforward;
the result for chromatic number requires the following purely combinatorial
result. We note that the proof is constructive, so gives an easy algorithm
for colouring the enhanced power graph.

\begin{theorem}
For every natural number $n$, there exist subsets $A_1,A_2,\ldots,A_n$ of
$\{1,2,\ldots,n\}$ with the properties
\begin{itemize}
\item $|A_q|=\phi(q)$ for $q\in\{1,\ldots,n\}$, where $\phi$ is Euler's
totient;
\item if $\lcm(q,q')\le n$, then $A_q\cap A_{q'}=\emptyset$, where $\lcm$
denotes the least common multiple.
\end{itemize}
\label{t:lem}
\end{theorem}

These theorems will be proved in the next two sections. In the final section
we give some concluding remarks.

Many further properties of power graphs and enhanced power graphs can be 
found in \cite{acst} and \cite{zbm}.

\section{Proof of Theorem~\ref{t:lem}}
Let $D$ be the set of fractions $p/q$ (in their lowest terms) in $(0,1]$,
for $1\le q\le n$. We define a function$f:D\to\{1,2,\ldots,n\}$ by the rule
\[f(p/q)=\lceil np/q\rceil.\]
The key observation is the following:
\begin{quote}
If $p/q\ne p'/q'$ and $f(p/q)=f(p'/q')$, then $\lcm(q,q')>n$.
\end{quote}
For, if $f(p/q)=f('p/q')$, then there exists $m$ such that
\[m-1<np/q,np'/q'\le m.\]
Thus $|p/q-p'/q'|<1/n$. On the other hand, $|p/q-p'/q'|$ is a rational
number whose numerator is at least $1$ (since $p/q\ne p'/q'$), and the
denominator is $\lcm(q,q')$. So we have
\[\frac{1}{n}>\left|\frac{p}{q}-\frac{p'}{q'}\right|\ge\frac{1}{\lcm(q,q')},\]
and so $\lcm(q,q')>n$, as required.

Now we let $D_q$ be the set of fractions in $D$ with denominator $q$, so that
$|D_q|=\phi(q)$, and let $A_q=f(D_q)\subseteq\{1,\ldots,n\}$. By our key
observation we see that
\begin{itemize}
\item the restriction of $f$ to $D_q$ is injective, so $|A_q|=\phi(q)$;
\item if $q\ne q'$ and $\lcm(q,q')\le n$, then $A_q\cap A_{q'}=\emptyset$.
\end{itemize}
So the theorem is proved.

\medskip

For example, here are the sets generated for $n=12$ by the above procedure.
\[\begin{array}{l}
A_1=\{12\}, A_2=\{6\}, A_3=\{4, 8\}, A_4=\{3, 9\}, \\
A_5=\{ 3, 5, 8, 10\}, A_6=\{ 2, 10\},  A_7= \{2, 4, 6, 7, 9, 11\}, \\
A_8=\{2, 5, 8, 11\}, A_9= \{2, 3, 6, 7, 10, 11\}, A_{10}= \{2, 4, 9, 11\}, \\
A_{11}=\{2, 3, 4, 5, 6, 7, 8, 9, 10, 11\}, A_{12}=\{1, 5, 7, 11\}.
\end{array}\]

\section{Proof of Theorem~\ref{t:main}}

We begin with the observation that if a finite set of elements in a group has
the property that any two of its elements generate a cyclic group, then the
whole set generates a cyclic group. A proof can be found in 
\cite[Lemma 32]{aetal}. It follows that a maximal clique in the enhanced power
graph is a maximal cyclic subgroup of $G$, and the clique number is equal to
the order of the largest cyclic subgroup, say $n$.

In order to find a colouring with $n$ colours, we take $\{1,2,\ldots,n\}$ to
be the set of colours, with the subsets $A_q$ given by Theorem~\ref{t:lem}.
We will use the set $A_q$ to colour elements of order $q$. If two elements of
order $q$ are joined, they lie in the same cyclic subgroup of order~$q$;
this subgroup has $\phi(q)$ generators, so we have enough colours to give them
all different colours. Other elements of order $q$ are not joined to these
ones, so we may re-use the same set of colours for them. Now, if two elements
of different orders $q$ and $q'$ are joined, they generate a cyclic group of
order $\lcm(q,q')$, which is at most $n$; so the sets of colours assigned to
them are disjoint. Thus, we obtain a proper colouring.

\section{Further remarks}

Our combinatorial lemma can deal with any set of element orders, as long as
the largest order $n$ is given. Now there are groups in which the set of
element orders is $\{1,\ldots,n\}$ for some $n$. (For example, the orders of
elements in the alternating group $A_7$ are $1,2,3,4,5,6,7$.) But, as we show
below, this can only occur for finitely many values of $n$. So, at first
glance, it seems we may be able to simplify the argument for most groups by
using the fact that not all orders occur. We have not attempted to do so, and
indeed it seems unlikely that any simplification can be obtained.

\begin{prop}
There are only finitely many values of $n$ for which there exists a finite
group in which the set of element orders is $\{1,\ldots,n\}$.
\end{prop}

\begin{proof}
We use the \emph{Gruenberg--Kegel graph} of a group $G$ (sometimes called the
\emph{prime graph}): the vertices are the prime divisors of $|G|$, with
vertices $p$ and $q$ joined if $G$ contains an element of order $pq$. Gruenberg
and Kegel described this graph in an unpublished manuscript on the 
decomposition of the augmentation ideal of the group ring; their main
theorem, a description of the groups whose Gruenberg--Kegel graph is
disconnected, was published by Gruenberg's student Williams~\cite{williams}
and refined by later authors, notably Kondrat'ev~\cite{kondratev}.

We will use the fact that the number of connected components of this graph is
at most $6$, for any finite group.

Now suppose that $G$ is a group in which the element orders are
$\{1,2,\ldots,n\}$. If $p$ is a prime in the interval $(n/2,n]$, then $p$ is
an isolated vertex in the Gruenberg--Kegel graph of $G$; so there can be at
most five such primes. But, in a strengthening of Bertrand's postulate,
Erd\H{o}s~\cite{erdos} showed that the number of primes in this interval tends
to $\infty$ with $n$. The result is proved.
\end{proof}

The weak perfect graph theorem asserts that a graph is perfect if and only if
its complement is perfect. This does not hold for weakly perfect graphs.
However, we note that Jitender Kumar Parveen has recently posted on the arXiv
a paper showing (among other things) that the complement of the enhanced
power graph of a finite group is weakly perfect~\cite{parveen}.

\medskip

A related graph is the difference of the power graph and enhanced power graph
of the group $G$, which we will denote by $\Delta(G)$: $x$ and $y$ are joined
in this graph if they are joined in the enhanced power graph but not in the
power graph.

For a group $G$, let $\Omega(G)$ denote the set of orders of elements of $G$.
For a positive integer $n$, let $\alpha(n)$ denote the size of the largest
antichain in the lattice of divisors of $n$. De Bruijn \emph{et al.}~\cite{db}
showed that, if $n$ has $m$ prime factors (counted with multiplicity), then
a maximum-size antichain consists of all divisors with $m/2$ prime factors
if $m$ is even, and either all divisors with $\lfloor m/2\rfloor$ prime
factors or all with $\lceil m/2\rceil$ prime factors if $m$ is odd. (This is
a generalisation of Sperner's lemma.)

\begin{prop}
For a finite group $G$, the clique number of $\Delta(G)$ is equal to
$\max\{\alpha(n):n\in\Omega(G)\}$.
\end{prop}

\begin{proof}
A clique $S$ in $\Delta(G)$ is a clique in the enhanced power graph, and so
is contained in a cyclic group $C$. Now a cyclic group has the property that
if $x$ and $y$ are two elements for which the order of $x$ divides the order
of $y$, then $x$ is a power of $y$. It follows that the elements of $S$ all
have different orders, and these form an antichain in the lattice of divisors
of $|C|$.
\end{proof}

\begin{prop}
Let $G$ be the symmetric group $S_8$ on $8$ letters. Then $\Delta(G)$ is not
weakly perfect.
\end{prop}

\begin{proof}
We have $\Omega(G)=\{1,2,3,4,5,6,7,8,10,12,15\}$; so the clique number of
$\Delta(G)$ is equal to $2$. But $\Delta(G)$ is not bipartite, since
\[\{(1,2), (3,4,5), (6,7), (1,2,3), (4,5,6,7,8)\}\]
induces a $5$-cycle.
\end{proof}

It is an interesting problem to describe the groups $G$ for which $\Delta(G)$
is weakly perfect, but we shall not discuss this here.

\end{document}